\documentclass[12pt]{amsart}

\usepackage{amsmath,bm}
\usepackage{hyperref}
\usepackage[nameinlink,capitalize]{cleveref}
\usepackage{dsfont}
\usepackage[english]{babel}
\usepackage{xcolor}

\newcommand{\summe}[2]{\displaystyle\sum_{#1}^{#2}}
\newcommand{\summezwei}[2]{\sum_{#1}^{#2}}

\newcommand{\integral}[4]{\displaystyle\int\limits_{#1}^{#2}{#3}\,{#4}}

\newcommand{\indikator}[2]{\mathds{1}_{{#1}}({#2})}
\newcommand{\indikatorzwei}[1]{\mathds{1}_{{#1}}}

\newcommand{\limes}[2]{\lim \limits_{#1 \rightarrow #2}}

\newcommand{\B}{\mathcal{B}}

\newcommand{\R}{\mathbb{R}}

\newcommand{\N}{\mathbb{N}}

\newcommand{\Pro}{\mathbb{P} } 
\newcommand{\eps}{\varepsilon} 
\newcommand{\secret}[1]{}

\newcommand{\norm}[1]{\| {#1}\|}

\newcommand{\bignorm}[1]{ \left \Vert {#1} \right \Vert}

\newcommand{\fmax}[2]{\text{$\overset{{#2} \textcolor{white}{h}}{\underset{{#1} \textcolor{white}{h}}{\bigvee\nolimits_f}} \; $}}
\newcommand{\fmaxzwei}[2]{\text{$\overset{{#2} \textcolor{white}{h}}{\underset{{#1} \textcolor{white}{h}}{\bigvee^{*}\nolimits_f}} \; $}}

\newtheorem{theorem}{Theorem}[section]
\newtheorem{lemma}[theorem]{Lemma}
\newtheorem{prop}[theorem]{Proposition}

\theoremstyle{definition}
\newtheorem{defi}[theorem]{Definition}

\newtheorem{example}[theorem]{Example}

\theoremstyle{remark}
\newtheorem{remark}[theorem]{Remark}

\numberwithin{equation}{section}
\begin{document}
\sloppy
\title[]{Implicit max-stable extremal integrals}

\author{D. Kremer}
\address{Dustin Kremer, Department Mathematik, Universit\"at Siegen, 57068 Siegen, Germany}
\email{kremer\@@{}mathematik.uni-siegen.de}

\date{\today}


\begin{abstract}
Recently, the notion of implicit extreme value distributions has been established, which is based on a given loss function $f \ge 0$. From an application point of view, one is rather interested in extreme loss events that occur relative to $f$ than in the corresponding extreme values itself. In this context, so-called $f$-implicit $\alpha$-Fr\'{echet} max-stable distributions arise and have been used to construct independently scattered sup-measures that possess such margins. In this paper we solve an open problem in \cite{goldbach} by developing a stochastic integral of a deterministic function $g\ge 0$ with respect to implicit max-stable sup-measures. The resulting theory covers the construction of max-stable extremal integrals (see \cite{stoev}) and, at the same time, reveals striking parallels. 
\end{abstract}
\maketitle
\section{Introduction} \label{section1}
The theory of implicit extreme values is highly motivated by application, such as hydrology (see  the introductory example in \cite{implicit}), and tries to analyze the circumstances in which several impact factors lead to extreme loss or damage. Hence, different from classical extreme value theory (shortly: EVT), this perspective is less interested in the attained extreme values than in the study of complex systems that cause these extreme values. Particularly, the isolated impacts (components) of the system do not have to be extreme in any sense, but can still contribute to such \textit{extreme loss events}. \newline 
In this context it is reasonable to assume that the connection between the impact factors and the related loss is known. More precisely, throughout the paper we consider a fixed function $f: \R^d \rightarrow [0, \infty)$ that serves as some kind of \textit{loss function}, depending on $d \ge 1$ impact factors. For technical reasons, we have to assume that $f$ fulfills the following properties, which still appear natural for most examples: 
\begin{itemize}
\item[(i)] $f$ is continuous.
\item[(ii)] $f(x)=0$ if and only if $x=0$.
\item[(iii)] $f$ is $1$-homogeneous, i.e. we have $f(\lambda x)=\lambda f(x)$ for every $\lambda \ge 0$ and $x \in \R$.
\end{itemize}
Turning over to probability, we consider a random vector $X=(X^{(1)},...,X^{(d)})$ modeling the joint behavior of the $d$ impact factors. Then, for a sequence $(X_j)_{j \in \N}$ of identically distributed and independent (i.i.d.) random vectors, a major subject of classical multivariate EVT is to understand the asymptotic behavior (under possible normalization) of
\begin{equation} \label{eq:19081902}
M_k:= \max_{j=1,...,k}  X_j := \bigvee \limits_{j=1}^k X_j
\end{equation}
as $k \rightarrow \infty$, where the maximum is meant component-wise. Sometimes, one is also interested in the study of $\max_{j=1,...k} f(X_j)$, which leads to an associated univariate problem. \newline
In contrast and as motivated before already, we want to pursue an implicit approach instead. Thus, if we assume for a moment that the \textit{observations} $X_1,X_2,...$ of the sample do not coincide, this suggests to consider
\begin{equation} \label{eq:19081901}
X_{j(k)}:= \text{argmax}_{j=1,...,k} \, f(X_j), \quad k \in \N.
\end{equation} 
Unfortunately, there will be ties in general. For this reason we replace the argmax function by the following \textit{$\vee_f$-operation}, which has been introduced in \cite{goldbach}. Hence, let $\vee_f: \R^d \times \R^d \rightarrow \R^d$, defined by
\begin{equation} \label{eq:25071981}
\vee_f(x_1,x_2):= x_1 \vee_f x_2:=  \begin{cases} x_1, & \text{if $f(x_1) \le f(x_2)$} \\ x_2, & \text{if $f(x_1)< f(x_2)$.} \end{cases}
\end{equation}
Inductively, for $x_1,...,x_k \in \R^d$, we define
\begin{equation*}
\fmax{j=1}{k} x_j :=\fmax{1\le j \le k}{} x_j :=  x_1 \vee_f \cdots \vee_f x_k:= (x_1 \vee_f \cdots x_{k-1}) \vee_f x_k.
\end{equation*}
Note that the resulting mapping is $\B((\R^d)^k)$-$\B(\R^d)$ measurable (see Lemma 1.1.6 in \cite{goldbach}), where $\B(\R^d)$ denotes the collection of all Borel sets $A \subset \R^d$, and that the $\vee_f$-operation is associative, but not commutative in general. However, it is a main feature of the $\vee_f$-operation that the result is always part of the sample (in contrast to \eqref{eq:19081902}). \newline
Anyway, \eqref{eq:19081901} can be rewritten as $X_{j(k)}=X_1 \vee_f \cdots \vee_f X_k$ now. Also note that the study of $a_k^{-1} X_{j(k)}$, where $a_k>0$ is suitable, and the characterization of possible limits in distribution (as $k \rightarrow \infty$) are the main topic of \cite{implicit}. The authors of \cite{implicit} call these limits \textit{implicit extreme value distributions}. There it is also shown that, under mild assumptions, the class of implicit extreme value distributions coincides with the class of \textit{implicit max-stable distributions}. Here, an $\R^d$-valued random vector $Y$ with distribution $\mu:=\mathcal{L}(Y)$ is said to be implicit max-stable if, for every $k \in \N$, there exists some $b_k>0$ such that 
\begin{equation} \label{eq:08071902}
b_k^{-1} \fmax{j=1}{k} Y_j=\fmax{j=1}{k} b_k^{-1}  Y_j \overset{d}{=}Y
\end{equation}
holds true, where $Y_1,Y_2,...$ is an i.i.d. sequence with common distribution $\mu$ and where $\overset{d}{=}$ denotes equality in distribution. The importance of stability equalities like \eqref{eq:08071902} is familiar, for example when using the \glqq$+$\grqq-operation or the \glqq$\vee$\grqq-operation on $\R^d$ instead of $\vee_f$. Then, stochastic processes whose margins possess such distributions are of interest and are often constructed by using certain stochastic integrals with respect to \textit{independently scattered random measures} (see \cite{SaTaq94}) or corresponding \textit{sup-measures} (see \cite{stoev}), respectively. In the last-mentioned case this leads to an \textit{extremal integral}, which is well-defined for every function $g$ that belongs to $ L^{\alpha}_{+}(m)$, where
\begin{equation} \label{eq:29071988}
 L^{\alpha}(m):= \left \{g: E \rightarrow \R \, | \, \text{$g$ is measurable with } \norm{g}_{\alpha}:= \left (\int_E |g|^{\alpha} \, dm \right)^{1 / \alpha}<\infty  \right\}
\end{equation}
and $ L^{\alpha}_{+}(m):=\{g:E \rightarrow \R_{+}  \, | \, g \in  L^{\alpha}(m)  \}$. Here, $\alpha>0$ is somehow connected to the underlying sup-measure. In particular, this diversity of possible \textit{integrands} $g$ allows the authors in \cite{stoev} to study a deep relationship between extremal integral representations and max-stable stochastic processes that are well-known in literature.  \newline
Hence, in \cite{goldbach} the notion of a so-called \textit{implicit sup-measure} is introduced (see \cref{15071901} below), which extends the sup-measures from \cite{stoev}. Actually, \cite{goldbach} even provides the existence of such implicit sup-measures, denoted by $M$ in the sequel. The details will be discussed in \cref{section2}. For the time being, we rather refer the reader to Example 3.1.15 in \cite{goldbach}, where $X(t):=M([0,t])$ leads to an $\R^d$-valued stochastic process $\mathbb{X}=\{X(t):t \ge 0\}$ that is implicit max-stable in the sense of \cref{12071904} below. Also note that \cite{goldbach} proposes the definition of a certain stochastic integral with respect to $M$, which allows for the representation 
\begin{equation} \label{eq:12071933}
X(t)=M([0,t]) = \integral{\R}{}{\indikator{[0,t]}{s}}{dM(s)}.
\end{equation}
Here, $g(s)=\indikator{[0,t]}{s}$ is a simple function in the sense of \eqref{eq:11071902} below. Unfortunately, the definition of the stochastic integral in \cite{goldbach} is just valid for simple functions $g \ge 0$ and is therefore more or less limited to considerations as in \eqref{eq:12071933}. However, based on the yields in \cite{stoev}, an extension of \eqref{eq:12071933} could be very interesting accordingly.  \newline
In effect, this paper mainly pursues two goals. One the one hand, we want to stimulate this new field of EVT (also see \cite{fondeville} and \cite{dombry}), which often allows us to recover results from classical EVT, see \cref{15071911} below. On the other hand, the subsequent findings could serve as a helpful tool to solve several problems that have already been discussed in literature. For instance, we think about the study of so-called \textit{$f$-implicit max-infinitely-divisible distributions} (see Definition 2.1.1 in \cite{goldbach}). At the same time and based on the asymptotic theory in \cite{implicit}, it might be tempting to construct implicit max-stable processes using the outcome in this paper.   \newline
It is also mentionable that we will obtain results that, at first glance, are similar to those in \cite{stoev}, where a (classical) max-stable extremal integral has been constructed. Somehow this also means that we have parallels to the notion of \textit{$\alpha$-stable} stochastic integrals, as proposed by \cite{SaTaq94}. However, our techniques are mostly different, since monotonicity arguments do not work any longer and since the $\vee_f$-operation can be rough sometimes. Thus, we believe that these techniques are of independent interest. \newline
The paper is organized as follows: In \cref{section2} we will provide a short review of the underlying concepts in order to understand \eqref{eq:12071933} in more detail. Then, in \cref{section3}, we will expand \eqref{eq:12071933} towards the notion of an (implicit) stochastic integral that, in the end, will be realized as a stochastic limit and that allows every function $g \in  L^{\alpha}_{+}(m)$ to serve as integrand. Finally, \cref{section4} is devoted to present some useful properties and examples, which emphasize the implicit approach of our theory. Nevertheless, they also demonstrate the intimate relation to the results in \cite{stoev}. 
\section{Preliminaries} \label{section2}
In this section we briefly want to introduce some notation and, by the way, recall from \cite{goldbach} and \cite{stoev} what we know so far about (implicit) $\alpha$-Fr\'{e}chet distributions and their corresponding sup-measures. Throughout let $(\Omega, \mathcal{A},\Pro)$ be some underlying probability space on which all occurring random elements are defined. \newline
As usual an $\R$-valued random variable $Z$ is said to be $\alpha$-Fr\'{e}chet distributed, where $\alpha>0$, if 
\begin{equation} \label{eq:17071901}
\Pro(Z \le x)= \begin{cases}   \exp(- \sigma^{\alpha} x^{-\alpha}), & x>0 \\ 0, & x \le 0.   \end{cases}
\end{equation}
Here, $\sigma \ge 0$ is referred to as the \textit{scale} (coefficient) of $Z$ and we write $Z \sim \Phi_{\alpha}(\sigma)$. Note that $\sigma=0$ leads to the point measure in zero, i.e. $\Phi_{\alpha}(0)=\eps_0$ for every $\alpha>0$. In the case $\sigma=1$ we say that $Z$ is \textit{standard} $\alpha$-Fr\'{e}chet distributed, abbreviated by $Z \sim \Phi_{\alpha}$. \\
Although Proposition 3.19 and Theorem 4.2 in \cite{implicit} provide a characterization of implicit max-stable distributions on $\R^d$, there exists no direct counterpart to the \textit{Fisher-Tippett-Gnedenko Theorem} (see \cite{extreme}). However, in a very natural special case, the solutions of \eqref{eq:08071902} lead to the following notion, which is due to Definition 3.1.2 in \cite{goldbach}. Recall that the function $f: \R^d \rightarrow \R_{+}$, fulfilling the properties (i)-(iii) from \cref{section1}, is fixed throughout the paper. 
\begin{defi} \label{23071920}
Fix $\alpha>0$ and let $\kappa$ be a probability measure on $\B(S)$, where $S:=\{f=1\}$ is a compact subset of $\R^d$. Then an $\R^d$-valued random vector $Y$ is said to have an \textit{$f$-implicit $\alpha$-Fr\'{e}chet distribution with scale $\sigma \ge 0$ and angular part $\kappa$} if 
\begin{equation} \label{eq:08071901}
Y \overset{d}{=} \sigma Z \Theta,
\end{equation} where the random variable $Z \sim \Phi_{\alpha}$ is independent of the $S$-valued random vector $\Theta$, being $\kappa$-distributed. We write $Y \sim \Phi_{\alpha, \kappa}^f(\sigma)$ in this case. Moreover, let $\Phi_{\alpha, \kappa}^f:=\Phi_{\alpha, \kappa}^f(1)$ again.
\end{defi}
Observe that \eqref{eq:08071901} is nothing else than $Y \overset{d}{=} Z_{\sigma} \Theta$ with $Z_{\sigma} \sim \Phi_{\alpha}(\sigma)$ and that $f(Y) \sim \Phi_{\alpha}(\sigma)$ in this case. Hence, the denomination \textit{$f$-implicit $\alpha$-Fr\'{e}chet} becomes reasonable. \newline
Now we can proceed with the consideration of independently scattered random measures and sup-measures, respectively. The first named ones have a long history, particularly in the context of $\alpha$-stable or, more generally, infinitely-divisible stochastic integrals and processes. See \cite{paper}, \cite{RajRo89} and \cite{SaTaq94}, just to mention a few. In this case, property $(RM_2)$ of the following definition essentially has to be modified by using the \glqq$+$\grqq-operation. In contrast, when using the \glqq$\vee$\grqq-operation instead (as established in \cite{stoev}), we are dealing with so-called (independently scattered) sup-measures. In the following we will refine this idea according to Definition 3.1.8 in \cite{goldbach}, where the set $L_0^d:=\{X:  \Omega \rightarrow \R^d \, | \, \text{$X$ is random vector}\}$ ($d \ge 1$) is of interest. 
\begin{defi} \label{15071901}
Let $(E, \mathcal{E},m)$ be a $\sigma$-finite measure space with $\mathcal{E}_0:=\{A \in \mathcal{E}: m(A)<\infty\}$. Then, for $f$ as before, a mapping $M^f: \mathcal{E}_0 \rightarrow L_0^d$ is called an \textit{$f$-implicit sup-measure} if the following conditions are fulfilled:
\begin{itemize}
\item[($RM_1$)] For finitely many sets $A_1,...,A_k \in \mathcal{E}_0$ the corresponding random vectors $M^f(A_1),...,M^f(A_k)$ are independent. 
\item[($RM_2$)] For any collection of disjoint sets $A_1,A_2,... \in \mathcal{E}_0$ such that $\cup_{j=1}^{\infty} A_j \in \mathcal{E}_0$ we have that
\begin{equation*}
M^f \left( \cup_{j=1}^{\infty} A_j\right)=\fmax{j=1}{\infty}M^f(A_j)=M^f(A_{j_0}) \quad \text{almost surely},
\end{equation*}
where $j_0$ is a random index.
\end{itemize}
In addition, if $M^f(A)$ has an $f$-implicit $\alpha$-Fr\'{e}chet distribution for every $A \in \mathcal{E}_0$ and some $\alpha>0$, the $f$-implicit sup-measure $M^f$ is said to be $\alpha$-Fr\'{e}chet.
\end{defi}
Of course, the question arises whether non-trivial examples of such sup-measures do exist. A satisfying answer is given by the following statement, which is due to Theorem 3.1.12 in \cite{goldbach}. Note that Definition 3.1.18 in \cite{goldbach} suggests an even more general idea. Yet we will not pursue this one in the sequel.
\begin{prop} \label{11070901}
Fix $\alpha>0$ and an arbitrary probability measure $\kappa$ on $\B(S)$. Then, for every $\sigma$-finite measure space $(E, \mathcal{E},m)$, there exists an $f$-implicit $\alpha$-Fr\'{e}chet sup-measure $M^f: \mathcal{E}_0 \rightarrow L_0^d$ such that $M^f(A) \sim \Phi_{\alpha,\kappa}^f(m(A)^{1/ \alpha})$ for every $A \in \mathcal{E}_0$.
\end{prop}
From now on and by a little abuse of notation, we will neglect the fact that the sup-measure, which is provided by \cref{11070901}, depends on $f, \alpha, \kappa$ and $m$. Hence, we merely abbreviate this sup-measure by $M$. Then a function $g:E \rightarrow \R_{+}=[0,\infty)$ is called \textit{simple (with respect to $\mathcal{E}_0$)} if, for some $k \in \N$, there exist $\alpha_1,...,\alpha_k \ge 0$ and disjoint sets $A_1,...,A_k \in \mathcal{E}_0$ such that the representation 
\begin{equation} \label{eq:11071902}
g(s)=\summe{j=1}{k} \alpha_j \indikator{A_j}{s} , \quad s \in E
\end{equation}
holds true. Certainly, such a representation is not unique. However, we get the following. 
\begin{defi} \label{12071910}
Let $g:E \rightarrow \R_+$ be simple and assume that a corresponding representation for $g$ is given by \eqref{eq:11071902}. Then the $\R^d$-valued random vector
\begin{equation} \label{eq:17071910}
I_{M}^{\vee_f}(g):=I(g):= \int_E^{\vee_f} g(s) \, dM(s) :=\fmax{j=1}{k} \alpha_j M(A_j)
\end{equation}
is uniquely determined by $g$ a.s. (see Proposition 3.2.3 in \cite{goldbach}), which means that $I(g)$ does not depend on a particular representation for $g$. We call $I(g)$ the \textit{($f$-implicit extremal) integral of $g$ (with respect to $M$)}. 
\end{defi}
Several properties of the ($f$-implicit extremal) integral $I(g)$ for simple functions $g$ can be found in Proposition 3.2.4 of \cite{goldbach}. Later, in the context of \cref{09051933}, we will study them into more detail. For the moment, it suffices just to mention the following one:
\begin{equation} \label{eq:29071939}
I(g) \sim \Phi_{\alpha, \kappa}^f (\norm{g}_{\alpha} ), \quad \text{where} \quad  \norm{g}_{\alpha} = \left( \integral{E}{}{g(s)^{\alpha}}{dm(s)}  \right)^{1/ \alpha}
\end{equation}
according to \eqref{eq:29071988}. Finally, Corollary 3.2.5 in \cite{goldbach} shows that $I(g)=I(\tilde{g})$ almost surely (a.s.), provided that $g$ and $\tilde{g}$ coincide $m$-almost everywhere (a.e.). This is one of the reasons why $m$ is often referred to as the \textit{control measure} of $M$.
\begin{remark} \label{23071901}
Recall that \cite{stoev} introduces an extremal integral for deterministic functions $g \ge 0$ with respect to certain $\alpha$-Fr\'{echet} sup-measures, which leads to $\R_{+}$-valued random variables. We omit the details. However, using Corollary 3.1.16 in \cite{goldbach} and the fact that
\begin{equation}  \label{eq:29071993}
f(x_1 \vee_f x_2)=f(x_1) \vee f(x_2) \quad \text{for any $x_1,x_2 \in \R^d$},
\end{equation}
it is easy to check that the random variable $f(I(g))$ equals the corresponding extremal integral in \cite{stoev}, provided that the integrand $g \ge 0$ is a simple function.
\end{remark}  
\section{Extension of the integral} \label{section3}
Recall that the sup-measure $M$ as well as the underlying ingredients $f, \alpha, \kappa$ and $m$ are fixed throughout. We want to start with a definition that appears not only general, but also natural in our context. In addition, it follows former examples in literature, which also deal with stochastic integrals (see \cite{paper}, \cite{RajRo89}, \cite{SaTaq94} and \cite{stoev} again). However, we will restrict ourselves to the consideration of $\R_{+}$-valued functions.
\begin{defi} \label{08051911}
A measurable function $g:E \rightarrow \R_{+}$ is called \textit{integrable with respect to $M$} (shortly: \textit{$M$-integrable}) if there exists a sequence of simple functions $(g_n)_n$ fulfilling $g_n \uparrow g$ such that the sequence $(I(g_n))_{n}$ converges in probability (on $\R^d$). Here, $g_n \uparrow g$ means that $g_n(s) \le g_{n+1}(s)$ for every $ n \in \N$ and $s \in E$, while $\sup_{n \in \N} g_n(s)= \lim_{n \rightarrow \infty} g_n(s)=g(s)$. \newline 
Finally, let $\mathcal{I}(M)$ denote the set of all functions $g:E \rightarrow \R_+$ that are $M$-integrable.
\end{defi}
The main aim of this section is to answer two questions that immediately arise from the previous definition:
\begin{itemize}
\item[(1)] Which functions belong to $\mathcal{I}(M)$?
\item[(2)] Given a function $g \in \mathcal{I}(M)$. Does the stochastic limit $I(g):=\Pro \text{-} \lim_{n \rightarrow \infty} I(g_n)$ depend on the choice of $(g_n)$? \newline
And if not, what are the properties of $I(g)$? (This will be the subject of \cref{section4}.)
\end{itemize}
We have seen in \cref{12071910} that the integral for simple functions essentially uses the $\vee_f$-operation. However, the pursued extension will also benefit from an operation that is quite related and that we will introduce now.
\begin{defi} \label{04041901}
For $k \ge 2$ and $x_1,...,x_k \in \R^d$ arbitrary let $j_0 \in \{1,...,n\}$ be the index such that $x_1 \vee_f \cdots \vee_f x_k=x_{j_0}$. Then we define
\begin{equation*}
\fmaxzwei{j=1}{k} x_j :=  \fmax{1 \le j \ne j_0 \le k}{} x_j. 
\end{equation*}
\end{defi}
The following observation combines both operations from a probabilistic point of view and, at the same time, reveals aspects from classical EVT. Therefore, recall \eqref{eq:17071901}. 
\begin{lemma} \label{03041905}
Assume that $X_1,...,X_k$ are $\R^d$-valued and independent random vectors, where $f(X_j) \sim \Phi_{\alpha}(\sigma_j)$ with scale $\sigma_j \ge 0$ for $j=1,...,k$. Then we have the following:
\begin{equation*}
\forall \, 0<\gamma <1: \quad \Pro \left(  f \left(\fmax{j=1}{k} X_j  \right)\le (1+\gamma) \, f \left( \fmaxzwei{j=1}{k} X_j  \right) \right) \le 1- (1+ \gamma)^{- \alpha}.
\end{equation*}
\end{lemma}
\begin{proof}
Without loss of generality we can assume that $\sigma_j>0$ for $j=1,...,k$. To start with a general observation, let $Y_1$ and $Y_2$ be independent random variables, where $Y_i \sim \Phi_{\alpha} (\sigma^{(i)})$ with scale $\sigma^{(i)} > 0$ for $i=1,2$. Fix $0<\gamma<1$. Then a standard calculation, using the substitution $1-\hat{\gamma}=(1+\gamma)^{- \alpha}$,  shows that
\begin{align}
\Pro(Y_1\le Y_2 \le (1+ \gamma) Y_1 ) &=\integral{0}{\infty}{ [ e^{- \sigma^{(2)} (1+\gamma)^{-\alpha} x^{-\alpha} } - e^{- \sigma^{(2)} x^{-\alpha} } ] \alpha \sigma^{(1)} x^{-\alpha -1} e^{- \sigma^{(1)}  x^{- \alpha}} }{dx} \notag \\
&=  \frac{\sigma^{(1)}}{\sigma^{(1)}  + (1+\gamma)^{- \alpha} \sigma^{(2)} } - \frac{\sigma^{(1)}}{\sigma^{(1)}+\sigma^{(2)}} \notag \\
&= \frac{ \hat{\gamma} \, \sigma^{(2)}  }{\sigma^{(1)} + \sigma^{(2)}  } \times  \frac{\sigma^{(1)} }{ \sigma^{(1)} + (1- \hat{\gamma}  )\sigma^{(2)}}  \notag \\
& \le  \frac{ \hat{\gamma} \, \sigma^{(2)}  }{\sigma^{(1)} + \sigma^{(2)}  } . \label{eq:03041907}
\end{align}
Now let $Y^{(i)}= \max_{1 \le j \ne i \le k} f(X_j)$ for $1 \le i \le k$. Then, by independence, we see on the one hand that $Y^{(i)}$ is independent from $f(X_i)$. On the other hand, $Y^{(i)}$ is also $\alpha$-Fr\'{e}chet distributed with scale $\sum_{1 \le j \ne i \le k} \sigma_j$. Moreover, note that 
\begin{equation} \label{eq:17071922}
 \left \{ f \left(\fmax{j=1}{k} X_j  \right)\le (1+\gamma) \, f \left( \fmaxzwei{j=1}{k} X_j \right) \right \}  = \bigcup_{i=1}^{k} \{Y^{(i)} \le f(X_i) \le (1+\gamma) Y^{(i)}  \}
\end{equation}
holds true. Then, for every $i \in \{1,..,k \}$, we can apply \eqref{eq:03041907} to $Y_1=Y^{(i)}$ and $Y_2=f(X_i)$, respectively. Using \eqref{eq:17071922}, this gives the assertion, since
\begin{equation*}
 \sum \limits_{i=1}^{k} \frac{ \hat{\gamma} \, \sigma_i  }{ \sum_{1 \le j \ne i \le k} \sigma_j + \sigma_i  } = \hat{\gamma}= 1- (1+ \gamma)^{- \alpha}.
\end{equation*}
\end{proof}
Remember that \eqref{eq:17071910} does not depend on a particular representation for $g$. However, a similar approach using the operation from \cref{04041901} would lead to serious problems. In order to overcome them, we need some additional notation.
\begin{defi}  \label{25041901}
\begin{itemize}
\item[(a)] Let $g$ be a simple function with representation $g=\summezwei{j \in I}{} \alpha_j \indikatorzwei{A_j}$ for some finite index set $I$, where $A_j \in \mathcal{E}_0$ are disjoint and where $\alpha_j \ge 0$ for every $j \in I$. Although we do not claim that $\cup_{j \in I} A_j=E$, we call $\{A_j: j \in I \}$ a \textit{partition} (of $g$) in this case and write $g \sim (A_j,\alpha_j)_{j \in I}$. 
\item[(b)] Assume that $\mathcal{P}_1$ and $\mathcal{P}_2$ are two partitions. Then we write $\mathcal{P}_1 \le \mathcal{P}_2$ if the following holds true: Any set from $\mathcal{P}_1$ can be represented by an appropriate union over sets belonging to $\mathcal{P}_2$.
\item[(c)] Consider a sequence $(g_n)$ of simple functions. Then a corresponding sequence of representations
\begin{equation} \label{eq:03041901}
g_n(s)=\summe{j=1}{k_n} \alpha_j^{(n)} \indikator{A_j^{(n)}}{s}, \quad s \in E
\end{equation}
is called \textit{consistent} if $\mathcal{P}_n \le \mathcal{P}_{n+1}$, where $\mathcal{P}_n:=\{A_1^{(n)},...,A_{k_n}^{(n)}\}$ for every $n \in \N$.
\end{itemize}
\end{defi}
Note that the following remark, in particular part (b), fixes the problem that we addressed above \cref{25041901}. However, its proofs are easy and therefore left to the reader.
\begin{remark} \label{03041902}
\begin{itemize}
\item[(a)] Suppose that $g_1,g_2$ are simple functions that can be represented by $g_i \sim (A_j^{(i)}, \alpha_j^{(i)})_{j=1,...,k_i}$ with $\mathcal{P}_i:=\{A_1^{(i)},...,A_{k_i}^{(i)} \}$ for $i=1,2$. Then we can always find a \textit{common partition} $\mathcal{P}$, which fulfills $\mathcal{P}_1, \mathcal{P}_2 \le \mathcal{P}$. More precisely, define
\begin{equation*}
A_0^{(1)}:=\bigcup_{j=1}^{k_2} A_j^{(2)} \setminus  \bigcup_{j=1}^{k_1} A_j^{(1)} , \quad  A_0^{(2)}:=\bigcup_{j=1}^{k_1} A_j^{(1)} \setminus  \bigcup_{j=1}^{k_2} A_j^{(2)}
\end{equation*}
and let $\alpha_0^{(1)}=\alpha_0^{(2)}=0$. Hence, we observe that
\begin{equation*}
\mathcal{P}= \{A_{j_1}^{(1)} \cap A_{j_2}^{(2)} : 0 \le j_1 \le k_1, 0 \le j_2 \le k_2\}
\end{equation*}
has the desired properties, since we can write
\begin{equation*}
g_i(s)= \summe{j_1=0}{k_1} \summe{j_2=0}{k_2} \alpha_{j_i}^{(i)} \indikator{A_{j_1}^{(1)} \cap A_{j_2}^{(2)}}{s}   \quad \text{for every $s \in E$ and $i=1,2$.}
\end{equation*}
\item[(b)] Consider a sequence $(g_n)$ of simple functions and assume that a consistent sequence of representations is given by \eqref{eq:03041901}, which is always possible due to part (a). Moreover, assume that $g_n \uparrow$, which means that the sequence $(g_n)$ itself is increasing. Then, using $(RM_2), $ it follows for every $n \in \N$ that
\begin{equation*}
f \left(  \fmaxzwei{j=1}{k_n} \alpha_j^{(n)} M(A_j^{(n)}) \right) \le  f \left(  \fmaxzwei{j=1}{k_{n+1}} \alpha_j^{(n+1)} M(A_j^{(n+1)}) \right) \quad \text{a.s.}
\end{equation*}
\end{itemize}
\end{remark}
Equipped with the previous observations, we are now able to enhance the idea of \cref{03041905}. Note that, for $\R$-valued functions $g,g_1,g_2,...$ on $E$, we shortly write $g_n \le g$, provided that $g_n(s) \le g(s)$ holds true for every $n \in \N$ and $s \in E$. Also recall \eqref{eq:29071988}.
\begin{prop} \label{25041910}
Let $(g_{n})$ be a sequence of simple functions fulfilling $g_n \le g$ for some $g \in L^{\alpha}_{+}(m)$ and assume that a consistent sequence of representations is given by \eqref{eq:03041901}. Define
\begin{equation*}
X_n:= I(g_n)=\fmax{j=1}{k_n} \alpha_j^{(n)} M(A_j^{(n)}) \quad \text{and}  \quad  X_n^{*}:=\fmaxzwei{j=1}{k_n} \alpha_j^{(n)} M(A_j^{(n)}).
\end{equation*}
Moreover, assume that there exist further sequences $(h_{1,n}), (h_{2,n})$ of simple functions such that $h_{1,n} \le g_n \le h_{2,n}$ and $h_{i,n} \uparrow g$ for $i=1,2$ as $n \rightarrow \infty$. Then, for any $\eps>0$, there exist a set $A \in \mathcal{A}$ with $\Pro(A) \ge 1-Ê\eps$ as well as some $\delta>0$ and $N \in \N$ such that we have
\begin{equation} \label{eq:18071933}
 f(X_n) (\omega) \ge (1+\delta) f(X_n^{*}) (\omega) \quad \text{for every $n \ge N$ and $\omega \in A$.}
 \end{equation}
\end{prop}
\begin{proof}
Obviously, we can always assume that $\emptyset$ is not an element of the occurring partitions. This allows to define $\beta_j^{(n)}:=\min \{h_{2,n}(s): s \in A_j^{(n)}\}$ and in view of $g_n \le h_{2,n}$ we obtain that $\alpha_j^{(n)} \le \beta_j^{(n)}$ for every $n \in \N$ and $1 \le j \le k_n$. If we let $h_n \sim (A_j^{(n)},\beta_j^{(n)})_{j=1,...,,k_n}$ together with
\begin{equation*}
 Y_n:=I(h_n)=\fmax{j=1}{k_n} \beta_j^{(n)} M(A_j^{(n)}) \quad \text{and} \quad Y_n^{*}:=\fmaxzwei{j=1}{k_n} \beta_j^{(n)} M(A_j^{(n)}),
\end{equation*}
it follows for every $n \in \N$ that $f(X_n^{*}) \le f(Y_n^{*}) $ a.s. Also note that the sequence $(h_n)$ is increasing, since the same holds true for $(h_{2,n})$ by assumption. In particular, we deduce that $f(Y_n^{*})$ is increasing due to part (b) of \cref{03041902}. Let $Y^{*}:= \sup_{n \in \N} f(Y_n^{*})$ and verify that $h_{1,n} \le g_n \le h_n \le h_{2,n}$. Then Proposition 3.2.4 (together with (1.3.2)) in \cite{goldbach} states that 
\begin{equation} \label{eq:12071901}
\forall n \in \N: \quad f(I(h_{1,n})) \le f(X_n) \le f(Y_n) \le f(I(h_{2,n})) \quad \text{a.s.}
\end{equation}
However, \cref{23071901} and Proposition 2.7 in \cite{stoev} imply that the increasing sequences $f(I(h_{1,n}))$ and $f(I(h_{2,n}))$ have the same limit a.s., say $Y$. It follows that $f(Y_n) \rightarrow Y$ a.s. Also note that $Y \ge Y^{*}$ and that $0<Y<\infty$ a.s., provided that $\norm{g}_{\alpha}>0$ (otherwise we conclude that $g_n=0$ $m$-a.e. and \eqref{eq:18071933} is true anyway).  \newline
The next step is to prove that $Y-Y^{*}>0$ holds true a.s. Conversely, assume that there exists a set $B \in \mathcal{A}$ with $p:=\Pro(B) >0$ and $Y(\omega)/Y^{*}(\omega)=1$ for every $\omega \in B$. Then we obtain that $(f(Y_n)/f(Y_n^{*})-1) \indikatorzwei{B} \rightarrow 0$ a.s. (and particularly in probability). Hence, for $\gamma, \gamma'>0$ arbitrary, it follows that
\begin{equation*}
 \Pro \left(  (f(Y_n)/f(Y_n^{*})-1) \indikatorzwei{B} \le \gamma \right) \ge 1 - \gamma' \quad \text{for almost all $n$,}
\end{equation*}
which also implies that $\Pro( f(Y_n)  \le (1 + \gamma) f(Y_n^{*}) ) \ge p -\gamma'$ for those $n$. Observe that this gives a contradiction to \cref{03041905}, when choosing $0<\gamma'<p+ (1+\gamma)^{-\alpha}-1$, which is always possible as long as we have that $p=1$ or $0<\gamma<(1-p)^{-1 / \alpha} - 1$, respectively.   \newline
Fix $\eps>0$. By what we have just seen there exist some $0< \delta'<1$ and a set $A_1 \in \mathcal{A}$ with $\Pro(A_1) \ge 1- \eps/2$, fulfilling the relation $Y\ge (1+ \delta') Y^{*} $ on $A_1$. In a similar way and using that $f(I(h_{1,n})) \uparrow Y$ a.s. (see above), we obtain some $N \in \N$ and a further set $A_2 \in \mathcal{A}$ with $\Pro(A_2) \ge 1- \eps/2$ and such that $f(I(h_{1,n}))(\omega) /Y(\omega) \ge 1- \delta' / 2$ holds true for every $\omega \in A_2$ and $n \ge N$. Let $A=A_1 \cap A_2$ and observe that $\Pro(A) \ge 1- \eps$. Finally, recall \eqref{eq:12071901} and that $f(X_n^{*}) \le f(Y_n^{*}) \le Y^{*}$. Then, for $n \ge N$, the following computation is valid on $A$, where we can assume that $f(X_n^{*})>0$ (else \eqref{eq:18071933} is true anyway again):
\begin{equation*}
\frac{f(X_n)}{f(X_n^{*})} \ge \frac{f(I(h_{1,n})) }{Y^{*}} = \frac{Y}{Y^{*}}  \cdot \frac{f(I(h_{1,n})) }{Y} \ge (1+\delta')(1-\delta' /2).
\end{equation*}
Letting $\delta:=(1+\delta')(1-\delta' /2) -1 >0$ for instance, this gives the assertion.
\end{proof}
Roughly speaking, a reformulation of \cref{25041910} states that we observe \textit{gaps} behind the attained maxima, which appear with a demanded probability and where the size of these gaps depends on the given probability. We will now try to benefit from these gaps and, therefore, handle some of the troubles that can be caused by the $\vee_f$-operation. Recall the set $S$ from \cref{23071920}.
\begin{lemma} \label{06051901}
Consider $\alpha_j, \beta_j>0$ and $x_j \in \R^d $ for $j=1,...,k$ and some $k \in \N$, where $\gamma:=\min\{\alpha_1,...,\alpha_k, \beta_1,....,\beta_k\}$ and $\rho:= \max \{|\alpha_j - \beta_j|: j=1,...,k\}$. Let 
\begin{equation*}
\zeta:= \fmax{j=1}{k} \alpha_j x_j \quad \text{and} \quad \zeta^{*}:= \fmaxzwei{j=1}{k} \alpha_j x_j.
\end{equation*}
Furthermore, assume that there exists some $\delta>0$ such that we have $f(\zeta) \ge (1+ \delta) f(\zeta^{*})$. Then, provided that $\rho < \gamma(\sqrt{1+ \delta}-1) $, the following relation holds true:
\begin{equation*}
\bignorm{\fmax{j=1}{k} \alpha_j x_j - \fmax{j=1}{k} \beta_j x_j} \le C \, \rho \, \max \limits_{j=1,...,k} f(x_j),
\end{equation*}
where $\norm{\cdot}$ denotes the Euclidean norm on $\R^d$ and where $C:= \max \{ \norm{x} : \text{$x \in S$}  \}$.
\end{lemma}
\begin{proof}
The case $\zeta=0$ is equivalent to $x_1=\cdots=x_k=0$ and therefore obvious. Else let $j_0$ be the index fulfilling $\zeta= \alpha_{j_0} x_{j_0}>0$, which particularly implies that $f(x_{j_0})>0$. We generally note that $\norm{x}=f(x) \norm{x / f(x) }$ holds true for every $x \in \R^d \setminus \{0\}$. This means that we have $\norm{x} \le C f(x)$ for every $x \in \R^d$. Thus, since $\norm{\cdot}$ is symmetric and since $f$ is $1$-homogeneous, the assertion would follow if we could show that $f(\beta_1 x_1) \vee_f \cdots \vee_f f(\beta_k x_k) = \beta_{j_0} x_{j_0}$. For this purpose, we first observe that
\begin{equation*}
\frac{\alpha_j}{\beta_j} =1 + \frac{\alpha_j-\beta_j}{\beta_j} \le 1 + \frac{\rho}{\gamma} \quad \text{and} \quad \frac{\beta_j}{\alpha_j}  \le 1 + \frac{\rho}{\gamma} \qquad (j=1,...,k)
\end{equation*}
holds true. Fix $j \ne j_0$. Then, using the given assumptions, we obtain that 
\allowdisplaybreaks
\begin{align*}
f(\beta_j x_j) & \le \left( 1 + \frac{\rho}{\gamma} \right) f (\alpha_j x_j) \\
& \le \left(1 + \frac{\rho}{\gamma} \right)(1+ \delta)^{-1}   f (\alpha_{j_0} x_{j_0}) \\
& \le \left(1 + \frac{\rho}{\gamma} \right)^2 (1+ \delta)^{-1}   f (\beta_{j_0} x_{j_0}) \\
& < f (\beta_{j_0} x_{j_0}),
\end{align*}
which completes the proof. Note that we benefited from $f(x_{j_0})>0$ in the last step.
\end{proof}
\begin{lemma} \label{25041905}
Let $h : E \rightarrow \R_{+}$ be measurable and assume that $(h_n)$ is a sequence of simple functions with $h_n \le h$ and such that $h_n$ converges to $h$ uniformly on $E$. Then there exist further sequences $(h_{1,n}), (h_{2,n})$ of simple functions with $h_{1,n} \le h_n \le h_{2,n}$ and such that $h_{i,n} \uparrow h$ for $i=1,2$ as $n \rightarrow \infty$.
\end{lemma}
\begin{proof}
By assumption we can find a strictly increasing sequence of naturals $(N_l)_l$ such that, for any $n \ge N_l$ and $s \in E$, we have $h(s)-h_n(s) \le 1/l$. In the case $n< N_1$ let $h_{1,n}:=0$. Else we define
\begin{equation*}
h_{1,n}:= \max \{0, \max \{h_1(s),...,h_n(s) \} - 1/l \}, \quad \text{if $N_l \le n < N_{l+1}$.}
\end{equation*}
Now it is easy to verify that this gives a sequence  $(h_{1,n})$ of simple functions as desired. Conversely, we can simply choose $h_{2,n}:=  \max \{h_1,...,h_n \} $ for every $n \inÊ\N$.
\end{proof}
The next result is the main step in order to extend the definition of the $f$-implicit extremal integral.
\begin{prop} \label{08051904}
Assume that $g \in L^{\alpha}_{+}(m)$ and that $(g_n)$ is a sequence of simple functions fulfilling $g_n \le g$ together with $g_n(s) \rightarrow g(s)$ for ($m$-almost) every $s \in E$. Then there exists a sequence of increasing sets $E_1,E_2,... \in \mathcal{E}_0$ with $m(E \setminus \bigcup_{l=1}^{\infty} E_l)=0$ and such that, for any $l \in \N$, the sequence $(I(g_n \indikatorzwei{E_l}))$ converges in probability as $n \rightarrow \infty$.  
\end{prop}
\begin{proof}
Since the measure $m$ is $\sigma$-finite, we can use Egorov's theorem (see Chapter VI, Exercise 3.1 in \cite{Els07}) to obtain increasing sets $E_1,E_2,... \in \mathcal{E}$ with $m(E \setminus \bigcup_{l=1}^{\infty} E_l)=0$ and such that, for any $l \in \N$, the convergence $g_n \indikatorzwei{E_l} \rightarrow g \indikatorzwei{E_l}$ holds uniformly as $n \rightarrow \infty$. Using the $\sigma$-finiteness of $m$ again and by a little abuse of notation, we can even assume that $(E_l) \subset \mathcal{E}_0$. Moreover, note that the proof of the present statement is obvious in the case $\norm{g}_{\alpha}=0$. Hence, without loss of generality, we can even assume that $\norm{g \indikatorzwei{E_l}}_{\alpha}  >0$ for every $l \in \N$. \\
Fix $l \in \N$ and consider the sequence $(g_n \indikatorzwei{E_l})_n$, where $g_n \indikatorzwei{E_l}$ is still simple. Let $\mathcal{P}_n'$ denote a partition of $g_n \indikatorzwei{E_l}$ for every $n \in \N$ and define $\mathcal{P}_1= \mathcal{P}_1'$. Then, using the construction from \cref{03041902} (a), we obtain a common partition for $g_1 \indikatorzwei{E_l}$ and $g_2 \indikatorzwei{E_l} $, denoted by $\mathcal{P}_2$ and which, in addition, fulfills $\mathcal{P}_1, \mathcal{P}_2' \le \mathcal{P}_2$. Based on $\mathcal{P}_2$ and $\mathcal{P}_3'$, we do the same to obtain $\mathcal{P}_3$. Inductively, this gives a sequence $(\mathcal{P}_n)$ of partitions such that, on the one hand, we have $\mathcal{P}_{n-1}, \mathcal{P}_{n}'  \le \mathcal{P}_{n}$. On the other hand, $g_{n-1} \indikatorzwei{E_l} $ and $g_{n} \indikatorzwei{E_l} $ can be both represented by using the common partition $\mathcal{P}_n$ for every $n \ge 2$. In particular, if we assume that $\mathcal{P}_n$ consists of $A_1^{(n)},...,A_{k_n}^{(n)} \in \mathcal{E}_0 \setminus \{\emptyset \}$ (which is always possible, see the proof of \cref{25041910}), there exist $\alpha_{1}^{(n)},..., \alpha_{k_n}^{(n)} \ge 0$ such that we have $g_n \indikatorzwei{E_l} \sim (A_j^{(n)},\alpha_j^{(n)})_{j=1,...,k_n} $, i.e.
\begin{equation} \label{eq:25041903}
g_n \indikatorzwei{E_l}  =\summezwei{j=1}{k_n} \alpha_j^{(n)} \indikatorzwei{A_j^{(n)}}, \quad n \in \N.
\end{equation}
At the same time, whenever $m>n$, the previous construction also allows us to find suitable coefficients $\beta_{m,1}^{(n)},..., \beta_{m,k_m}^{(n)} \ge 0$, which only depend on $\alpha_1^{(n)},....,\alpha_{k_n}^{(n)}$ and which fulfill
\begin{equation} \label{eq:06051904}
g_n \indikatorzwei{E_l}  =\summezwei{j=1}{k_m} \beta_{m,j}^{(n)} \indikatorzwei{A_j^{(m)}}.
\end{equation}
Based on \eqref{eq:25041903}, we define
\begin{equation*}
X_n:= I(g_n \indikatorzwei{E_l})= \fmax{j=1}{k_n} \alpha_j^{(n)} M(A_j^{(n)}) \quad \text{and}  \quad  X_n^{*}:=\fmaxzwei{j=1}{k_n} \alpha_j^{(n)} M(A_j^{(n)})
\end{equation*}
for every $n \in \N$. In view of \cref{25041905} (with $h:=g \indikatorzwei{E_l}$) we can apply \cref{25041910} (with $g_n \indikatorzwei{E_l}$ instead of $g_n$) to $X_n$ and $X_n^{*}$ in this case. Hence, for fixed $\eps>0$, there exist a set $A_0 \in \mathcal{A}$ with $\Pro(A_0) \ge 1-Ê\eps/3$ as well as some $\delta>0$ and $N_0 \in \N$ fulfilling
\begin{equation} \label{eq:26041902}
 f(X_n) (\omega) \ge (1+\delta) f(X_n^{*}) (\omega) \quad \text{for every $n \ge N_0$ and $\omega \in A_0$.}
 \end{equation}
The fundamental idea is to use \cref{06051901} now. However, its assumptions are not fulfilled yet. As a way out, recall the proof of \cref{25041910} and that, in a very similar way, $f(X_n)$ converges to a random variable $Y$ a.s. In addition, Proposition 2.7 in \cite{stoev} states that $Y \sim \Phi_{\alpha}(\norm{g \indikatorzwei{E_l}}_{\alpha})$. Moreover, since $\norm{g \indikatorzwei{E_l}}_{\alpha}>0$, we have that $Y>0$ a.s. If we combine both results, there exist a set $A_1 \in \mathcal{A}$ with $\Pro(A_1) \ge 1-Ê\eps/3 $ as well as some $\tau>0$ and $N_1 \in \N$ such that 
\begin{equation*}
 f(X_n) (\omega) \ge \tau  \quad \text{for every $n \ge N_1$ and $\omega \in A_1$.}
 \end{equation*}
Let $j_0=j_0(n,\omega)$ be the (random) index fulfilling $X_n(\omega)=\alpha^{(n)}_{j_0} M(A_{j_0}^{(n)})(\omega)$. Here, without loss of generality, we can assume that $A_j^{(n)} \subset E_l$ for every $n \in \N$ and $1 \le j \le k_n$. Using Proposition 3.2.4 in \cite{goldbach} again, this implies for those $j$ and $n$ that 
 \begin{equation} \label{eq:14081901}
 f(M(A_j^{(n)}))= f(I(\indikatorzwei{A_j^{(n)}})) \le  f(I(\indikatorzwei{E_l})) =f(M(E_l)) \quad \text{a.s.},
 \end{equation}
where $f(M(E_l)) \sim \Phi_{\alpha} (m(E_l)^{1/ \alpha})$ due to \eqref{eq:29071939}. Note that $m(E_l)< \infty$, since $E_l \in \mathcal{E}_0$. Hence, there finally exist a set $A_2 \in \mathcal{A}$ with $\Pro(A_2) \ge 1- \eps /3$ and some $K>0$ such that $f(M(E_l))(\omega) \le K$ for every $ \omega \in A_2$. Let $A:=A_0 \cap A_1 \cap A_2$ and observe that $\Pro(A) \ge 1- \eps$. Moreover, for any $\omega \in A$ and $n \ge N:= \max \{N_0,N_1\}$, we obtain that $\alpha_{j_0}^{(n)} \ge \tau/K$. Let 
\begin{equation*}
I_{n}=\{1 \le j \le k_n: \alpha_{j}^{(n)} \ge \tau / 2K\} \quad \text{and} \quad \widetilde{g_n} \sim (A_j^{(n)}, \alpha_j^{(n)}  )_{j \in I_n}.
\end{equation*}
Then it is clear that $X_n= I(g_n)$ and $I(\widetilde{g_n})$ coincide for every $n \ge N$ on $A$. In addition, if we introduce 
 \begin{equation} \label{eq:06051910}
Y_n:= I(\widetilde{g_n})=\fmax{j \in I_n}{} \alpha_j^{(n)} M(A_j^{(n)}) \quad \text{and}  \quad  Y_n^{*}:=\fmaxzwei{j \in I_n}{} \alpha_j^{(n)} M(A_j^{(n)}),
\end{equation}
relation \eqref{eq:26041902} can be preserved. More precisely, for every $n \ge N$ and $\omega \in A$, we have that
\begin{equation} \label{eq:25071901} 
 f(Y_n) (\omega) =  f(X_n) (\omega) \ge (1+\delta) f(X_n^{*}) (\omega)  \ge (1+\delta) f(Y_n^{*}) (\omega)  .
 \end{equation}
On the other hand, since the convergence $g_n \indikatorzwei{E_l} \rightarrow g \indikatorzwei{E_l}$ holds uniformly, we can choose some $N' \ge N$ such that, for every $m, n \ge N'$ and $s \in E$, the estimation
\begin{equation} \label{eq:06051911}
\norm{g_m \indikatorzwei{E_l} - g_n \indikatorzwei{E_l} }_{\infty}:= \sup_{s \in E} |g_m \indikator{E_l}{s} - g_n \indikator{E_l}{s} | < \min \left \{\frac{\tau}{4K} , \frac{\eps}{C K}, \frac{\tau}{4K} (\sqrt{1+ \delta}-1) \right\}
\end{equation}
is valid with $C$ being defined as in \cref{06051901}. Moreover, we claim that 
\begin{equation} \label{eq:06051902}
\forall m,n \ge N': \quad \Pro( \norm{I(g_m \indikatorzwei{E_l})- I(g_n \indikatorzwei{E_l}) } \ge \eps) \le \eps
\end{equation}
holds true. Recall that $\eps>0$ was arbitrary. Hence, it is well-known that \eqref{eq:06051902} would imply that the sequence $(I(g_n \indikatorzwei{E_l}))_n$ is Cauchy with respect to convergence in probability (see Corollary 6.15 in \cite{Kle08} for instance) and would therefore complete the proof. In order to prove \eqref{eq:06051902}, let us assume that $m>n \ge N'$ are fixed naturals. Since $\Pro(A) \ge 1- \eps$, it suffices to show for every $\omega \in A$ that $\norm{I(g_m \indikatorzwei{E_l})(\omega)- I(g_n \indikatorzwei{E_l}) (\omega)} < \eps$. For this purpose, we additionally fix $\omega \in A$ and recall that $X_m(\omega)=I(g_m \indikatorzwei{E_l})(\omega)= Y_m(\omega)$ according to \eqref{eq:06051910}. At the same time, we can also use another representation for $g_n \indikatorzwei{E_l}$, which is given by \eqref{eq:06051904}. More precisely, we have that $g_n \indikatorzwei{E_l} \sim (A_j^{(m)},\beta_{m,j}^{(n)} )_{j=1,...,k_m}$, where $\beta_{m,1}^{(n)},...,\beta_{m,k_m}^{(n)} \ge 0$ are appropriate coefficients. Recall that $\emptyset \notin \mathcal{P}_m$. Hence, by definition of $I_m$ and in view of \eqref{eq:06051911}, we obtain for every $j \in \{1,...,k_m\} \setminus I_m$ the estimation
\begin{equation} \label{eq:25071922}
\beta_{m,j}^{(n)} \le  |\alpha_{j}^{(m)} -  \beta_{m,j}^{(n)}|  + \alpha_{j}^{(m)}   <  \norm{g_m \indikatorzwei{E_l} - g_n \indikatorzwei{E_l} }_{\infty} + \frac{\tau}{2K} <  \frac{3 \tau}{4K} < \frac{\tau}{K}.
\end{equation}
Similarly to \eqref{eq:06051910}, the previous observation suggests to consider the truncation 
\begin{equation*}
Z_n:=Z_n^{(m)}:= \fmax{j \in I_m}{} \beta_{m,j}^{(n)} M(A_j^{(m)})
\end{equation*}
and to conclude that $I(g_n \indikatorzwei{E_l} ) (\omega)= Z_n(\omega)$. At this point, we neglect the fact that $I(g_n \indikatorzwei{E_l} )$ could vary on a $\Pro$-null set by using the representation from \eqref{eq:06051904} now. Anyway, let us summarize that the equality
\begin{equation} \label{eq:25071999}
\norm{I(g_m \indikatorzwei{E_l})(\omega)- I(g_n \indikatorzwei{E_l}) (\omega)}= \bignorm{\fmax{j \in I_m}{} \alpha_j^{(m)} M(A_j^{(m)})(\omega)   -   \fmax{j \in I_m}{} \beta_{m,j}^{(n)} M(A_j^{(m)})(\omega)  }
\end{equation}
holds true. Then, a similar calculation as performed in \eqref{eq:25071922}, using \eqref{eq:06051911} and the reverse triangle equality, ensures that $\beta_{m,j}^{(n)}, \alpha_j^{(m)} \ge  \frac{\tau}{4K}$ for every $j \in I_m$. Hence, if we let
\begin{equation*}
\gamma:=\frac{\tau}{4K} \quad \text{as well as} \quad  \rho:=\norm{g_m \indikatorzwei{E_l} - g_n \indikatorzwei{E_l} }_{\infty}
\end{equation*}
and recall \eqref{eq:14081901}, we can use \cref{06051901} together with \eqref{eq:25071901}-\eqref{eq:06051911} again to conclude that \eqref{eq:25071999} is smaller than $\eps$. As justified before already, this gives the assertion.
\end{proof}
Before putting things together, we have do deal with the \textit{remainder} $g_n \indikatorzwei{E_l^c}: =g_n \indikatorzwei{E \setminus E_l} $. 
\begin{lemma} \label{09051901}
In the situation of \cref{08051904} we have the following: For any $\eps>0$ there exists some $L=L(\eps) $ such that
\begin{equation}
\forall n ,l \ge L: \quad \Pro( I(g_n) \ne I(g_n \indikatorzwei{E_l})) \le \eps. \label{eq:12061901}
\end{equation}
\end{lemma}
\begin{proof}
Consider $n,l \in \N$. Since $g_n= \max \{g_n \indikatorzwei{E_l},g_n \indikatorzwei{E_l^c}\}$, Proposition 3.2.4 in \cite{goldbach} reveals that the random vectors $I(g_n \indikatorzwei{E_l})$ and $I(g_n \indikatorzwei{E_l^c})$ are independent and that 
\begin{equation} \label{eq:25071987}
I(g_n)=I(g_n \indikatorzwei{E_l}) \vee_f  I(g_n \indikatorzwei{E_l^c})= I(g_n \indikatorzwei{E_l^c})  \vee_f     I(g_n \indikatorzwei{E_l}) \quad \text{a.s.}
\end{equation} 
Recalling \eqref{eq:25071981}, we see that $I(g_n) \ne I(g_n \indikatorzwei{E_l})$ is equivalent to $f(I(g_n \indikatorzwei{E_l})) < f(I(g_n \indikatorzwei{E_l^c}))$ in this case and that \eqref{eq:12061901} would follow if we can prove that
\begin{equation} \label{eq:25071950}
\Pro(f(I(g_n \indikatorzwei{E_l})) < f(I(g_n \indikatorzwei{E_l^c}))) \rightarrow 0 \quad (\text{as $n,l \rightarrow \infty$}).
\end{equation}
Note that $f(I(g_n \indikatorzwei{E_l})) \sim \Phi_{\alpha}(\norm{g_n \indikatorzwei{E_l}}_{\alpha})$ and $f(I(g_n \indikatorzwei{E_l^c})) \sim \Phi_{\alpha}(\norm{g_n \indikatorzwei{E_l^c}}_{\alpha})$, respectively. On the one hand, this shows that we can assume that $\norm{g_n \indikatorzwei{E_l^c}}_{\alpha}>0$ (which particularly implies that $\norm{g}_{\alpha}>0)$. On the other hand, a similar computation as performed in \eqref{eq:03041907} yields 
\begin{equation*}
\Pro(f(I(g_n \indikatorzwei{E_k})) < f(I(g_n \indikatorzwei{E_k^c})))   = \frac{ \norm{g_n \indikatorzwei{E_l^c}}_{\alpha} }{\norm{g_n \indikatorzwei{E_l}}_{\alpha} + \norm{g_n \indikatorzwei{E_l^c}}_{\alpha} }  =\left( 1+ \frac{\norm{g_n \indikatorzwei{E_l}}_{\alpha}}{\norm{g_n \indikatorzwei{E_l^c}}_{\alpha}}  \right)^{-1} . 
\end{equation*}
Hence, instead of \eqref{eq:25071950} it suffices to show that 
\begin{equation} \label{eq:08051910}
1+ \frac{\norm{g_n \indikatorzwei{E_l}}_{\alpha}^{\alpha} }{\norm{g_n \indikatorzwei{E_l^c}}_{\alpha}^{\alpha}}  =  \frac{\int_{E} g_n^{\alpha} \, dm}{\int_{E_l^c} g_n^{\alpha} \, dm}   =\frac{\norm{g_n }_{\alpha}^{\alpha} }{ \norm{g_n \indikatorzwei{E_l^c}}_{\alpha}^{\alpha}}  \rightarrow \infty \quad (\text{as $n,l \rightarrow \infty$}).
\end{equation}
For this purpose, observe that we have $\norm{g_n }_{\alpha}^{\alpha} \rightarrow   \norm{g}_{\alpha}^{\alpha}>0$ (as $n \rightarrow \infty)$ by the dominated convergence theorem. Conversely, we obtain (for every $n \in \N$) that $ \norm{g_n \indikatorzwei{E_l^c}}_{\alpha}^{\alpha} \le  \norm{g \indikatorzwei{E_l^c}}_{\alpha}^{\alpha} \rightarrow 0$ (as $l \rightarrow \infty$), since $E_l^{c} \downarrow $ with $m(\cap_{l=1}^{\infty} E_l^{c})=0$ and since $g \in L^{\alpha}_{+}(m)$. This implies \eqref{eq:08051910}.

\end{proof}
Finally, we are able to answer the questions from the beginning of this section. In this context, recall \cref{08051911} and notice that \eqref{eq:08051920} below will respect the definition of $I(g)$ so far (see \eqref{eq:17071910}). Also note that the proof of part (b) of the following result benefits from the fact that we stated \cref{08051904} in an extensive way. That is we did not demand the convergence $g_n \rightarrow g$ to be monotone in the first place. 
\begin{theorem} \label{12061903}
We have the following:
\begin{itemize}
\item[(a)] $\mathcal{I}(M)=L^{\alpha}_{+}(m)$, which is independent of $f$.
\item[(b)] Assume that $g \in L^{\alpha}_{+}(m)$ and that $(g_n)$ is a sequence of simple functions fulfilling $g_n \uparrow g$. Then the sequence $(I(g_n))$ converges in probability and this limit does a.s. not depend on the particular choice of simple functions $(g_n)$ with $g_n \uparrow g$.
\end{itemize}
We call this limit the \textit{($f$-implicit extremal) integral of $g$ (with respect to $M$)} and write
\begin{equation} \label{eq:08051920}
I_{M}^{\vee_f}(g):=I(g):= \int_E^{\vee_f} g(s) \, dM(s) :=\Pro\text{-} \limes{n}{\infty} I(g_n). 
\end{equation}
\end{theorem}
\begin{proof}
We first prove part (b): Fix $\eps>0$ and consider increasing sets $E_1,E_2,... \in \mathcal{E}_0$ as provided by \cref{08051904}. According to \cref{09051901} there exists some $L$ such that we have $\Pro( I(g_n) \ne I(g_n \indikatorzwei{E_l})) \le \eps/3$ for every $n,l \ge L$. At the same time, \cref{08051904} states that $(I(g_n \indikatorzwei{E_{L}}))_n$ is Cauchy (in probability), i.e. we can find some $N$ fulfilling 
\begin{equation*}
\Pro(\norm{I(g_m \indikatorzwei{E_L})  - I(g_n \indikatorzwei{E_L})} \ge \eps/3   ) \le \eps/3 \quad \text{for every $m,n \ge N$}. 
\end{equation*}
Note that the event $\{\norm{I(g_m )-I(g_n)} \ge  \eps  \} $ is a subset of 
\begin{equation} \label{eq:30071936}
 \{  \norm{I(g_m \indikatorzwei{E_L})-I(g_n \indikatorzwei{E_L} )} \ge \eps/3 \} \,  \cup  \bigcup_{i \in \{ m,n\}} \{ \norm{I(g_i)-I(g_i \indikatorzwei{E_L})} \ge \eps/3 \}.
 \end{equation}
Hence, for every $m,n \ge \max \{L,N\}$, we easily conclude that $\Pro(\norm{I(g_m)-I(g_n)} \ge  \eps ) \le \eps$, which, as in the proof of \cref{08051904}, shows that $(I(g_n))$ converges in probability. Denote this limit by $X$ and consider a different sequence of simple functions $(g_n')$, still fulfilling $g_n' \uparrow g$. Then, repeating the previous ideas, we obtain that $(I(g_n'))$ converges in probability, say with limit $X'$. Define a further sequence of simple functions $(h_{\nu})_{\nu \in \N}$ by $h_{2 n-1}:=g_{n}$ and $h_{2 n}:=g'_{n}$ for every $n \in \N$, respectively. In particular, we observe that \cref{08051904} as well as \cref{09051901} still apply to $(h_{\nu})$ such that $(I(h_{\nu}))$ converges, too. However, by regarding suitable subsequences, it follows that $X$ and $X'$ coincide a.s. \newline
Concerning part (a), consider $g \in L^{\alpha}_{+}(m)$ and choose a sequence $(g_n)$ of simple functions with $g_n \uparrow g$ (see section 2.3 in \cite{stoev} for example to verify that such a sequence always exists). By what we have just proved, it follows that $g \in \mathcal{I}(M)$ and therefore that $ L^{\alpha}_{+}(m)\subset \mathcal{I}(M) $. Conversely, fix $g \in \mathcal{I}(M)$ and let $(g_n)$ be a proper sequence of simple functions in the sense of \cref{08051911}. Denote the associated stochastic limit of $(I(g_n))$ by $Y$ and observe that we have $\norm{g_n}_{\alpha} \uparrow \norm{g}_{\alpha}  \in [0,\infty]$ by the monotone convergence theorem. At the same time, 
\begin{equation} \label{eq:25071947}
\forall x>0: \quad \Pro(f(I(g_n)) \le x) =  \exp(- \norm{g_n}_{\alpha}^{\alpha}  \, x^{-\alpha})
\end{equation}
holds true, while Proposition 3.2.4 in \cite{goldbach} implies that $(f(I(g_n)))$ is increasing a.s. However, by the continuous mapping theorem, the corresponding limit coincides with $f(Y)$ a.s. In view of  of \eqref{eq:25071947} and since $f(Y)$ is $[0,\infty)$-valued, it is easy to check that $ \norm{g}_{\alpha}< \infty$, i.e. we have that $g \in L^{\alpha}_{+}(m)$ and therefore $ L^{\alpha}_{+}(m)\supset \mathcal{I}(M) $.
\end{proof}
\section{Properties and examples} \label{section4}
Based on \cref{12061903}, it appears natural to study properties of the mapping $\mathcal{I}(M) \ni g \mapsto I(g)$ in the sequel. Actually, we already encountered some of them, for example in \eqref{eq:25071987}. A closer look on \eqref{eq:25071987} reveals that, at least for simple functions $g$, the stochastic integral manages to overcome some of the problems that occur in the context of the $\vee_f$-operation. It will be crucial to gain a corresponding insight for functions $g \in L^{\alpha}_{+}(m)\ $. Therefore, we start with the following preparation. 
\begin{lemma} \label{01071910}
Let $g_1,g_2 \in L^{\alpha}_{+}(m)$ such that $g_1 \le g_2$. Then $I(g_1) \vee_f I(g_2)= I(g_2) \vee_f I(g_1)$ holds true a.s.
\end{lemma}
\begin{proof}
Recall \eqref{eq:25071981} and the beginning of the proof of \cref{09051901}. Then, letting 
\begin{equation} \label{eq:29071999}
A:=\{ I(g_1)Ê\ne I(g_2) \text{ and } f(I(g_1))=f(I(g_2)) \},
\end{equation} 
we have to show that $\Pro(A)=0$. Since $g_1 \le g_2$, there exist sequences $(g_{1,n})$ and $(g_{2,n})$ of simple functions such that $g_{1,n} \le g_{2,n}$ and $g_{i,n} \uparrow g_i$ for $i=1,2$ as $n \rightarrow \infty$. Moreover, \cref{03041902} allows us to find a common sequence of partitions (each not containing $\emptyset$, see above) for $g_{1,n}$ and $g_{2,n}$, which, in addition, is consistent. More precisely, let us  assume that
\begin{equation*}
g_{1,n} \sim (A_j^{(n)}, \alpha_j^{(n)})_{j=1,...,k_n} \quad \text{and} \quad  g_{2,n} \sim (A_j^{(n)}, \beta_j^{(n)})_{j=1,...,k_n}, 
\end{equation*}
respectively. In view of $g_{1,n} \le g_{2,n}$ and $A_j^{(n)} \ne \emptyset$, we necessarily have that $\alpha_j^{(n)} \le \beta_j^{(n)}$ for every $n \in \N$ and $1 \le j \le k_n$. Let
\begin{equation*}
X_n:=I(g_{1,n})= \fmax{j=1}{k_n} \alpha_j^{(n)} M(A_j^{(n)})  \quad \text{and} \quad Y_n:= I(g_{2,n})=\fmax{j=1}{k_n} \beta_j^{(n)} M(A_j^{(n)}) 
\end{equation*} 
together with
\begin{equation*} 
X_n^{*}:=\fmaxzwei{j=1}{k_n} \alpha_j^{(n)} M(A_j^{(n)}) \quad \text{and} \quad Y_n^{*}:=\fmaxzwei{j=1}{k_n} \beta_j^{(n)} M(A_j^{(n)}).
\end{equation*}
Anyway, \cref{12061903} states that $I(g_{i,n})$ converges to $I(g_i)$ in probability and therefore, by passing to a suitable subsequence, a.s. Without loss of generality, we omit the consideration of this subsequence in the sequel and therefore obtain a set $B \in \mathcal{A}$ such that $\Pro(B)=1$ and 
\begin{equation} \label{eq:29071901}
\forall \omega \in B \, \, \forall i=1,2: \quad I(g_{i,{n}}) (\omega) \rightarrow I(g_i)(\omega) \quad (n \rightarrow \infty).
\end{equation}
Now, if we assume that $\Pro(A)=:p>0$, we can apply \cref{25041910} to $(Y_n)$, providing a set $C \in \mathcal{A}$ with $\Pro(C) \ge 1-Êp/2$ as well as some $\delta>0$ and $N \in \N$ fulfilling
\begin{equation}  \label{eq:19071903}
 f(Y_n) (\omega) \ge (1+\delta) f(Y_n^{*}) (\omega) \quad \text{for every $n \ge N$ and $\omega \in C$.}
 \end{equation}
Note that, for certain (random) indices $j_1=j_1(n, \omega)$ and $j_2=j_2(n, \omega)$, we can always write 
\begin{equation*}
X_n(\omega)= \alpha_{j_1}^{(n)} M(A_{j_1}^{(n)}) \quad \text{and} \quad Y_n(\omega)= \beta_{j_2}^{(n)} M(A_{j_2}^{(n)}).
\end{equation*}
Moreover, observe that $\Pro(A \cap B \cap C) >0$. Then, for fixed $\omega \in A \cap B \cap C$, we have to distinguish two cases. In the first case, the indices $j_1$ and $j_2$ differ. Then, using \eqref{eq:19071903} and $\alpha_j^{(n)} \le \beta_j^{(n)}$, we obtain for every $n \ge N$ that 
\begin{equation} \label{eq:29071902}
f(X_n)(\omega) \le f(Y_n^{*}) (\omega) \le (1+ \delta)^{-1} f(Y_n) (\omega).
\end{equation}
However, by definition of the set $A$ and using the continuity of $f$ together with \eqref{eq:29071901}, we verify that $f(X_n)(\omega)/f(Y_n)(\omega) \rightarrow 1$. This means that \eqref{eq:29071902} can only happen for finitely many $n$. Else, the second case occurs, where $j_1=j_2$.  By the $1$-homogeneity of $f$, this yields
\begin{equation*} 
f( X_n(\omega) - Y_n(\omega) )=f(X_n(\omega) )- f(Y_n(\omega) ) \quad \text{for almost all $n$}.
\end{equation*} 
Using similar arguments as before, it follows that $f( X_n(\omega) - Y_n(\omega) ) \rightarrow 0$. However, in view of Lemma 3.1.14 in \cite{goldbach}, this implies that $(X_n(\omega) - Y_n(\omega)) \rightarrow 0$. Remembering that 
\begin{equation*}
I(g_1)(\omega)-I(g_2)(\omega)=\lim_{n \rightarrow \infty}(X_n(\omega) - Y_n(\omega)  )  \quad (n \rightarrow \infty),
\end{equation*}
we finally obtain that $I(g_1)(\omega)=I(g_1)(\omega) $, which is a contradiction to the claim $\omega \in A$.
\end{proof}
As announced before already, we want to proceed with some illuminating properties of the $f$-implicit extremal integral that mostly extend from the consideration of simple functions. In this context, recall from \cite{goldbach} the \textit{partial order} $\le_f$ on $\R^d$, being defined by
\begin{equation} \label{eq:29071953}
x \le_f y \qquad : \Leftrightarrow \qquad f(x)<f(y) \text{   \, or \,   } x=y.
\end{equation}
See Proposition 1.3.2 and Lemma 1.3.3 in \cite{goldbach} for several properties concerning this binary relation. Also note that the proof of \cref{01071910}, in particular the set $A$ from \eqref{eq:29071999}, already anticipated this relation to some extent. This means that we have $x \le_f y$ or $y \le_f x$ if and only if $x \vee_f y =y \vee_f x$. 
\begin{prop} \label{09051933}
Let $g_1,g_2 \in L^{\alpha}_{+}(m)$. 
\begin{itemize}
\item[(i)] ($f$-implicit $\alpha$-Fr\'{e}chet) The random vector $I(g_1)$ is $f$-implicit $\alpha$-Fr\'{e}chet distributed in the sense of \cref{23071920}. More precisely, $I(g_1) \sim \Phi_{\alpha, \kappa}^f (\norm{g_1}_{\alpha})$.
\item[(ii)] ($f$-implicit max-linearity) For $a,b \ge 0$ we have that
\begin{equation} \label{eq:01071999}
I(ag_1 \vee b g_2) = a I(g_1) \vee_f b I(g_2) \quad \text{a.s.},
\end{equation}
which particularly means that $I(g_1)$ and $I(g_2)$ commute under the $\vee_f$-operation.
\item[(iii)] ($f$-implicit independence) The random vectors $I(g_1)$ and $I(g_2)$ are independent if and only if $g_1g_2=0$ $m$-a.e.
\item[(iv)] ($f$-implicit monotonicity) We have: $g_1 \le g_2$ $m$-a.e. if and only if $I(g_1) \le_f  I(g_2) $ a.s. In addition, $g_1=g_2$ $m$-a.e. is equivalent to $I(g_1)= I(g_2)$ a.s. 
\end{itemize}
\end{prop}
\begin{proof}
For simple functions $g_1$ and $g_2$, the whole statement follows from Proposition 3.2.4 and Corollary 3.2.5 in \cite{goldbach}, respectively. We will use this fact without explicit reference in the sequel. Moreover and without loss of generality, we can assume that $\norm{g_i}_{\alpha}>0$ for $i=1,2$. Throughout let $(g_{1,n})$ and $(g_{2,n})$ be sequences of simple functions such that $g_{i,n} \uparrow g_i$ for $i=1,2$ and $n \rightarrow \infty.$ In particular, \cref{12061903} states that 
\begin{equation} \label{eq:29071955}
I(g_{i,n}) \rightarrow I(g_i) \quad \text{in probability \quad (for $i=1,2$ as $n \rightarrow \infty)$.}
\end{equation}
\begin{itemize}
\item[(i)] Since we have that $I(g_{1,n}) \sim \Phi_{\alpha, \kappa}^f (\norm{g_{1,n}}_{\alpha})$, while $\norm{g_{1,n}}_{\alpha} \uparrow \norm{g_{1}}_{\alpha}$ by the monotone convergence theorem, the assertion follows from \eqref{eq:29071955} by passing through the limit. 
\item[(ii)] Obviously, the homogeneity property $I(ag_1)=a I(g_1)$ extends from the consideration of simple functions to the present case. Therefore it suffices to consider the case $a=b=1$ in the following. Then, since $g_{1,n} \vee g_{2,n} \uparrow g_1 \vee g_2 \in L^{\alpha}_{+}(m) $, we derive from \cref{12061903} together with the accuracy of \eqref{eq:01071999} for simple functions that
\begin{equation*}
I(g_1 \vee g_2)= \Pro \text{-} \limes{n}{\infty} I(g_{1,n} \vee g_{2,n}) = \Pro \text{-} \limes{n}{\infty} \left(I(g_{1,n}) \vee_f I(g_{2,n}) \right).
\end{equation*}
At this point, recall \eqref{eq:29071955} and note that $I(g_{1,n}) \vee_f I(g_{2,n}) \in \{I(g_{1,n}), I(g_{2,n}) \}$ for every $n \in \N$. Hence, we need that the $\vee_f$-operation is continuous, which is not true in general (see Example 1.1.4 in \cite{goldbach}). However, in order to ensure continuity in our situation (and therefore to obtain the assertion), we merely need that 
\begin{equation} \label{eq:29071957}
I(g_{1}) \vee_f I(g_{2})=I(g_{2}) \vee_f I(g_{1}) \quad \text{a.s.} 
\end{equation}
To prove \eqref{eq:29071957}, we first consider the case $g_1g_2=0$ $m$-a.e. Then, without loss of generality, we can also assume that $g_{1,n}g_{2,n}=0$ holds true $m$-a.e, which means that $I(g_{1,n})$ and $I(g_{2,n})$ are independent for every $n \in \N$. On the one hand, it is clear that the corresponding stochastic limits, namely $I(g_1)$ and $I(g_2)$, preserve this property. Hence, $f(I(g_1))$ and $f(I(g_2))$ are independent, too. On the other hand, we have that $f(I(g_i)) \sim \Phi_{\alpha}(\norm{g_i}_{\alpha})$ due to part (i), which means that $f(I(g_1)) \ne f(I(g_2))$ a.s. (essentially use \eqref{eq:03041907} for $\gamma=0$). In particular, \eqref{eq:29071957} is fulfilled, provided that $g_1g_2=0$ $m$-a.e. Moreover, \cref{01071910} states that \eqref{eq:29071957} is still true as long as we have $g_1 \le g_2$. Finally, writing $g_i= g_i \indikatorzwei{g_1 \le g_2} \vee \indikatorzwei{g_1 > g_2}$ for general $g_i \in L^{\alpha}_{+}(m)$ and using the associativity of the $\vee_f$-operation, we can combine both observations to derive that
\begin{align*}
I(g_1) \vee_f I(g_2)& = I(g_1 \indikatorzwei{g_1 \le g_2} ) \vee_f I(g_1 \indikatorzwei{g_1 > g_2} ) \vee_f I(g_2 \indikatorzwei{g_1 \le g_2} ) \vee_f I(g_2 \indikatorzwei{g_1 > g_2} ) \\
&= I(g_1 \indikatorzwei{g_1 \le g_2} )  \vee_f I(g_2 \indikatorzwei{g_1 \le g_2} ) \vee_f I(g_1 \indikatorzwei{g_1 > g_2} ) \vee_f I(g_2 \indikatorzwei{g_1 > g_2} ) \\
&= I(g_2 \indikatorzwei{g_1 \le g_2} ) \vee_f I(g_1 \indikatorzwei{g_1 \le g_2} )   \vee_f I(g_2 \indikatorzwei{g_1 > g_2} )  \vee_f I(g_1 \indikatorzwei{g_1 > g_2} ) \\
&= I(g_2 \indikatorzwei{g_1 \le g_2} ) \vee_f I(g_2 \indikatorzwei{g_1 > g_2} )  \vee_f I(g_1 \indikatorzwei{g_1 \le g_2} )    \vee_f I(g_1 \indikatorzwei{g_1 > g_2} ) \\
&=I(g_2) \vee_f I(g_1)
\end{align*}
holds true a.s., which shows \eqref{eq:29071957}.
\item[(iii)] The \textit{if}-part turns out to be a by-product of the proof of part (ii) before. Conversely, assume that $I(g_1)$ and $I(g_2)$ are independent. Then, although $g_1$ and $g_2$ are not necessarily simple functions, properties (i) and (ii) allow to imitate the according part in the proof of Proposition 3.2.4 (iv) in \cite{goldbach} to conclude that $g_1g_2=0$ $m$-a.e.
\item[(iv)] Let us first prove the \textit{only if}-part, where we can assume that $g_{1,n} \le g_{2,n}$ again. It follows that $f(I(g_{1,n})) \le f(I(g_{2,n}))$ a.s. and hence that $f(I(g_1)) \le f(I(g_2))$ a.s. In view of \eqref{eq:29071953} this would already imply that $I(g_1) \le_f I(g_2)$ a.s., provided that $\Pro(A)=0$ holds true, where the set $A$ is defined as in \eqref{eq:29071999}. Actually, this was just the outcome of the proof of \cref{01071910}. \newline 
Conversely, if $I(g_1) \le_f I(g_2)$ a.s., we can exactly use the idea that has been presented in the proof of Proposition 3.2.4 (iii) in \cite{goldbach} to obtain that $g_1 \le g_2$ $m$-a.e., since this only uses the properties (i) and (ii) again. For the additional statement of part (iv), merely note the following observation:
\begin{equation*} 
\forall x,y \in \R^d: \qquad  x = y \quad  \Leftrightarrow \quad x \le_f y \text{ \,  and \, } y \le_f x.
\end{equation*}
\end{itemize}
\end{proof}
The next result characterizes the convergence in probability of the occurring stochastic integrals, namely in terms of the corresponding deterministic integrands that belong to $L^{\alpha}_{+}(m)$. As a by-product, it also shows that any sequence of approximating functions $(g_n)$ can be used in \eqref{eq:08051920} to reach $I(g)$ as long as one of the conditions in \eqref{eq:09051933} below holds true. More precisely, we are neither restricted to simple functions nor to monotone sequences anymore.
\begin{theorem}  \label{12061911}
Consider $g,g_1,g_2,...L^{\alpha}_{+}(m)$. Then, as $n \rightarrow \infty$, we have:
\begin{equation} \label{eq:09051933}
I(g_n) \overset{\Pro}{\longrightarrow} I(g) \quad \Leftrightarrow \quad \int_E |g_n^{\alpha}-g^{\alpha}| \, dm \rightarrow 0  \quad \Leftrightarrow \quad \int_E |g_n-g|^{\alpha} \, dm \rightarrow 0.
\end{equation}
\end{theorem}
Before proving \cref{12061911}, we need two auxiliary results, where the first one deals with the deliverance from monotone sequences that we announced already before.
\begin{lemma} \label{12061920}
Consider $g \in L^{\alpha}_{+}(m)$ and assume that $(g_n)$ is a sequence of simple functions fulfilling $g_n \le g$ together with $g_n(s) \rightarrow g(s)$ for ($m$-almost) every $s \in E$. Then we obtain that $I(g_n)\rightarrow I(g)$ in probability as $n \rightarrow \infty$.
\end{lemma}
\begin{proof}
Let $(g_n')$ be a sequence of simple functions fulfilling $g_n' \uparrow g$, which particularly means that $I(g)=\Pro$-$\lim_{n \rightarrow \infty} I(g_n')$. As in the proof of \cref{12061903}, define a new sequence $(h_{\nu})_{\nu \in \N}$ that alternates between $(g_n)$ and $(g_n')$. Again, it follows that $(I(h_{\nu}))$ converges in probability. Actually, this gives the assertion, since all subsequences yield the same limit. More precisely, 
\begin{equation*}
\Pro \text{-}\limes{n}{\infty} I(g_n)=  \Pro \text{-}\limes{n}{\infty} I(h_{2n-1}) = \Pro \text{-}\limes{n}{\infty} I(h_{2n})=\Pro \text{-}\limes{n}{\infty} I(g_n')= I(g).
\end{equation*}
\end{proof}
Note that the following observation could be stated in a more general framework. However, by doing it this way, it will allow an easy application within the proof of \cref{12061911} below.
\begin{lemma} \label{01081901}
Let $(h_{1,n}) \subset L^{\alpha}_{+}(m) $ fulfill $X=\Pro$-$\lim_{n \rightarrow \infty} I(h_{1,n})$ for some random vector $X$. Assume that $(h_{2,n})$ is a further sequence of functions such that we have $0 \le h_{2,n} \le \gamma_n \indikatorzwei{A}$ for every $n \in \N$, where $A \in \mathcal{E}_0$ and where $(\gamma_n) \subset \R_+$ converges to zero. Then we have that $I(h_{1,n} \vee h_{2,n}) \rightarrow X$ in probability, too (as $n \rightarrow \infty$).
\end{lemma}
\begin{proof}
Using homogeneity and the $f$-implicit monotonicity from \cref{09051933}, we first obtain that $f(I(h_{2,n})) \le \gamma_n f(M(A))$ a.s. (recall \eqref{eq:29071953}), which shows that $f(I(h_{2,n})) \rightarrow 0$ a.s. In view of Lemma 3.1.14 in \cite{goldbach} this also implies that $I(h_{2,n}) \rightarrow 0$ a.s. Moreover, since $I(h_{1,n} \vee h_{2,n})=I(h_{1,n}) \vee_f I(h_{2,n})$ a.s. due to \eqref{eq:01071999}, we merely need that the $\vee_f$-operation provides continuity a.s. For this purpose, recall the proof of the $f$-implicit max-linearity above or use Lemma 1.1.9 in \cite{goldbach}, respectively. 
\end{proof}
\begin{proof}[Proof of \cref{12061911}] Note that the last-mentioned equivalence in \eqref{eq:09051933} corresponds to Lemma 2.3 in \cite{stoev}. Hence, if we recall \eqref{eq:29071988}, the present proof reduces to the following:
\begin{equation} \label{eq:30071966}
I(g_n) \overset{\Pro}{\longrightarrow} I(g) \quad \Leftrightarrow \quad \norm{g_n^{\alpha}-g^{\alpha}}_1 \rightarrow 0 \qquad \text{(as $n \rightarrow \infty$).}
\end{equation}
Observe that, in the case $\norm{g}_{\alpha}=0$, we have $I(g)=0$ a.s. together with $\norm{g_n^{\alpha}-g^{\alpha}}_1=\norm{g_n}_{\alpha}$. Since \cref{09051933} implies that $I(g_n) \sim \Phi_{\alpha, \kappa}^f (\norm{g_n}_{\alpha})$ (see \eqref{eq:08071901}), it is easy to verify that \eqref{eq:30071966} holds true in this case. Thus, let us assume that $0<\norm{g}_{\alpha} < \infty$ in the sequel. \newline
Then, in order to prove \eqref{eq:30071966}, we first suppose that $I(g_n) \rightarrow I(g)$ in probability, which also implies that $f(I(g_n)) \rightarrow f(I(g))$ by the continuous mapping theorem. Recall that $f(I(g_n)) \sim \Phi_{\alpha}(\norm{g_n}_{\alpha})$ and $f(I(g)) \sim \Phi_{\alpha}(\norm{g}_{\alpha})$, respectively. Hence, a combination of \eqref{eq:29071993} and \eqref{eq:01071999} shows that
\begin{equation*}
f(I(g_n)) \vee f(I(g))= f(I(g_n) \vee_f I(g)) = f(I (g_n \vee g)) \quad a.s.
\end{equation*}
Accordingly, it follows that $f(I(g_n) \vee f(I(g))  \sim \Phi_{\alpha} (\norm{g_n \vee g}_{\alpha})$. Based on this, we can mostly follow the proof of Theorem 2.1 in \cite{stoev} to obtain that $\norm{g_n^{\alpha}-g^{\alpha}}_1 \rightarrow 0$. The details are left to the reader.  \\
Conversely, assume that $\norm{g_n^{\alpha}-g^{\alpha}}_1 \rightarrow 0$ holds true. Unfortunately, this merely implies that there exists a suitable subsequence along which $g_n$ converges to $g$ ($m$-a.e.). At the same time, it only remains to prove that $I(g_n) \rightarrow I(g)$ in probability, which can be characterized in the following way (use Theorem 20.5 in \cite{billingsley} for instance): Each subsequence of $(I(g_n))$ contains a further subsequence that converges to $I(g)$ in probability. Keeping this in mind, the following steps will reveal that, without loss of generality, we can already assume that $g_n \rightarrow g$ $m$-a.e as $n \rightarrow \infty$. Now fix $\eps>0$. Then we can use Egorov's Theorem again to obtain a sequence of increasing sets $E_1',E_2',... \in \mathcal{E}$ with $m (E \setminus \cup_{l=1}^{\infty} E_l') =0 $ and such that, for every $l \in \N$, the convergence $g_n \indikatorzwei{E_l'} \rightarrow g \indikatorzwei{E_l'}$ holds uniformly. Clearly, the previous observation remains true for $E_1,E_2,...$ (instead of $E_1',E_2',...)$, defined by 
\begin{equation*} 
E_l:=E_l ' \cap ( \{1/l \le g \le l \} \cup \{g=0\}), \quad l \in \N.
\end{equation*}
Note that $\norm{g_n }_{\alpha}^{\alpha} \rightarrow   \norm{g}_{\alpha}^{\alpha}>0$ by assumption. Using this together with \cref{09051933} and the fact that, for any $n,l \in \N$, the estimation
\begin{equation*}
\norm{g_n \indikatorzwei{E_l^c}}_{\alpha}^{\alpha}  \le \norm{g_n^{\alpha}-g^{\alpha}}_{1}^{\alpha}  + \norm{g \indikatorzwei{E_l^c}}_{\alpha}^{\alpha} 
\end{equation*}
is valid, we can argue as in the proof of \cref{09051901} to verify that \eqref{eq:12061901} is fulfilled accordingly. Obviously, the argument includes the consideration of the function $g$. More precisely, if we let $g_0:=g$ for the moment, there exists some $L=L(\eps)$ such that
\begin{equation*}
\forall n \in \{0,L,L+1,...\} \, \, \forall  l \ge L: \quad \Pro( I(g_n) \ne I(g_n \indikatorzwei{E_l})) \le   \eps/3.
\end{equation*}
If we use the triangular inequality (compare \eqref{eq:30071936}), it follows for any $n \ge L$ that
\begin{equation*}
\Pro(\norm{I(g_n)-I(g)} \ge  \eps)  \le \frac{2 \eps}{3 } + \Pro(\norm{I(g_n \indikatorzwei{E_{L}})-I(g \indikatorzwei{E_{L}})} \ge  \eps /3 ).
\end{equation*}
Hence, as already argued elsewhere, it suffices to show that the following relation holds true:
\begin{equation} \label{eq:30071978}
\Pro(\norm{I(g_n \indikatorzwei{E_{L}})-I(g \indikatorzwei{E_{L}})} \ge  \eps/3) \le \frac{\eps}{3}  \quad \text{for almost all $n$}. 
\end{equation}
For every $n \in \N$, let $(g_{n, \nu})_{\nu}$ be a sequence of simple functions with $g_{n, \nu} \uparrow g_n$ and such that the convergence $g_{n, \nu} \indikatorzwei{E_{L}} \rightarrow g_n \indikatorzwei{E_{L}}$ holds uniformly (as $ \nu \rightarrow \infty$) . Note that this is possible, since $g \indikatorzwei{E_L} \le L$, which means that $g_n \indikatorzwei{E_{L}}$ is also bounded (at least for almost all $n$). Recall the notation $\norm{\cdot}_{\infty}$ from \eqref{eq:06051911}. Then, according to \cref{12061903}, we can even find a strictly increasing sequence $(\nu(n))_n$ of naturals such that, for those $n \in \N$, 
\begin{equation*}
\norm{g_n \indikatorzwei{E_L}   - g_{n,\nu(n)} \indikatorzwei{E_L}    }_{\infty} \le 1/n   \quad \text{and} \quad \Pro(\norm{  I(g_{n} \indikatorzwei{E_{L}})  -  I(g_{n,\nu(n)} \indikatorzwei{E_{L}})  } \ge 1/n) \le 1/n  
\end{equation*}
hold true. In other words, as $n \rightarrow \infty$, we have that
\begin{equation} \label{eq:24061901}
  \norm{g_n \indikatorzwei{E_L}   - g_{n,\nu(n)} \indikatorzwei{E_L}    }_{\infty} \rightarrow 0 \quad \text{and} \quad I(g_{n} \indikatorzwei{E_{L}})  -  I(g_{n,\nu(n)} \indikatorzwei{E_{L}}) \overset{\Pro}{\longrightarrow} 0.
\end{equation}
In addition, we define the sequence $(\eta_n)$ by
\begin{equation*}
\eta_n:=\frac{L^{-1}}{L^{-1} + \norm{ g_n \indikatorzwei{E_{L}  }  -g \indikatorzwei{E_{L} } }_{\infty} } \quad \in (0,1]
\end{equation*}
and observe that, for every  $n \in \N$ and $s \in E_{L} \setminus \{g=0\} $, the following calculation is valid: 
\begin{equation*}
 g_n (s) \le \frac{g(s) +  \norm{ g_n \indikatorzwei{E_{L}  }  -g \indikatorzwei{E_{L} } }_{\infty} }{g(s)} g(s) \le \left (1+ \frac{ \norm{ g_n \indikatorzwei{E_{L}  }  -g \indikatorzwei{E_{L} } }_{\infty} }{L^{-1}} \right ) g(s) =\eta_n^{-1} g(s).
\end{equation*}
In view of $g_{n,\nu(n)} \le g_ n$ this shows that $h_{1,n}:= \eta_n \, g_{n,\nu(n)} \indikatorzwei{E_{L} \setminus \{g=0\}  } $ defines a simple function fulfilling $h_{1,n} \le g \indikatorzwei{E_L}$ for every $n \in \N$. Combine \eqref{eq:24061901} with $\norm{ g_n \indikatorzwei{E_{L} }  -g \indikatorzwei{E_{L} } }_{\infty} \rightarrow 0$ (see above) and note that $\eta_n \rightarrow 1$ to verify that $\norm{h_{1,n} - g \indikatorzwei{E_L} }_{\infty} \rightarrow 0$. In particular, \cref{12061920} implies that $I(h_{1,n}) \rightarrow I(g \indikatorzwei{E_{L}})$ in probability. Finally, let $h_{2,n}:=\eta_n \, g_{n,\nu(n)} \indikatorzwei{E_{L} \cap \{g=0\}  } $ and observe that 
\begin{equation} \label{eq:01081910}
h_n:= h_{1,n} \vee h_{2,n} =  \eta_n \, g_{n,\nu(n)} \indikatorzwei{E_{L}}, \quad n \in \N.
\end{equation}
Using $g_{n,\nu(n)} \le g_ n$ again, we also conclude that $h_{2,n} \le \norm{ g_n \indikatorzwei{E_{L} }  -g \indikatorzwei{E_{L} } }_{\infty} \, \indikatorzwei{E_{L} \cap \{g=0\}}$ holds true. Hence, the assumptions of \cref{01081901} are fulfilled and we obtain that $I(h_n) \rightarrow I(g \indikatorzwei{E_{L}})$ in probability. Finally, we benefit from the intimate relation between $g_{n, \nu(n)}$ and $h(n)$. More precisely, \eqref{eq:01081910} and the homogeneity of the stochastic integral lead to the fact that
\begin{equation*} 
I(g_{n,v(n)} \indikatorzwei{E_{L}}) - I(h_n) =(\eta_n^{-1}-1) I(h_n) \rightarrow 0 \cdot I(g \indikatorzwei{E_{L}}) = 0 
\end{equation*}
in probability as $n \rightarrow \infty$. If we write
\begin{equation*}
I(g_n \indikatorzwei{E_{L}}) -I(g \indikatorzwei{E_{L}})= I(g_n \indikatorzwei{E_{L}}) - I(g_{n,\nu(n)} \indikatorzwei{E_{L}}) + I(g_{n,\nu(n)} \indikatorzwei{E_{L}}) - I(h_n) + I(h_n) - I(g \indikatorzwei{E_{L}})
\end{equation*}
and use the previous outcome, if follows that $I(g_n \indikatorzwei{E_{L}}) -I(g \indikatorzwei{E_{L}}) \rightarrow 0$ in probability, which particularly implies the accuracy of \eqref{eq:30071978}. This completes the proof.
\end{proof}
\begin{remark}
For $g:E \rightarrow \R$ we can write $g=g^{+}-g^{-}$, where $g^{\pm}=\max\{0, \pm g\}$. Then, assuming that $g \in L^{\alpha}(m)$ (see \eqref{eq:29071988}), the classical approach would be to define $I(g)$ by $I(g^{+})-I(g^{-})$. However, this would be very debilitating with regard to \cref{09051933} and the desired properties of the stochastic integral. Actually, it could be interesting to use the definition $I(g)=I(g^{+}) \vee_f I(g^{-})$ instead. \newline
Apart from this and according to \cite{lixiao} for example, it could also be nice to consider matrix-valued functions $g$ as integrand. In this context, view \cite{bms} for the notion of so-called \textit{$B$-homogeneous} functions (where $B$ is a $d \times d$ matrix), which turns out to be an extension of $1$-homogeneous functions. But honestly, things will certainly become much more complicated in this case, since \cref{06051901} will no longer work, just to mention one reason.
\end{remark}
Let us also remark that there is no intuitive counterpart to Proposition 2.8 in \cite{stoev} (even if we use the $\le_f$ order). At this point, we finish the discussion about general properties of the $f$-implicit extremal integral. In the sequel we rather want to illustrate possible benefits of the theory that we recently developed. 
One aspect is that we retrieve the max-stable extremal integral that has been constructed in \cite{stoev}, leading to univariate random variables. For this purpose and as already indicated in the proof of \cref{25041910}, we merely have to manipulate our approach by considering $L_{+}^{\alpha}(m) \ni g \mapsto f(I(g))$ instead. On the other hand, there is also a straight possibility to do so, which means that we should finally talk about concrete choices of the loss function $f$.
\begin{example} \label{15071911}
Obviously, every norm on $\R^d$ can serve as loss function. However, for the rest of this example, let us consider the special case $d=1$ with 
\begin{equation*}
f_0=|\cdot| \quad \text{and} \quad S_0=\{f_0=1\}=\{-1,1\}. 
\end{equation*}
Then, for $x_1,x_2 \ge 0$ (which is the typical setting in the context of classical EVT), we see that 
\begin{equation} \label{eq:02081901}
x_1 \vee_{f_0} x_2= x_1 \vee x_2 \quad \text{as well as} \quad x_1 \le_{f_0} x_2 \Leftrightarrow x_1 \le x_2.
\end{equation}
At the same time, letting $\kappa=\eps_1$ in the context of \cref{11070901}, it follows that $M^{f_0}(A) \sim \Phi_{\alpha}(m(A)^{1/ \alpha})$. In particular, we have that $M^{f_0}(A) \ge 0$ a.s. for every $A \in \mathcal{E}_0$. It follows that the observation \eqref{eq:02081901} remains accordingly true for the occurring stochastic integrals, at least a.s. For instance note that \eqref{eq:01071999} becomes
\begin{equation*}
I(ag_1 \vee b g_2) = a I(g_1) \vee b I(g_2) \quad \text{a.s.}
\end{equation*}
in this case, which is just the so-called \textit{max-linearity} in the sense of \cite{stoev}. In general, it turns out that \cref{12061903}, \cref{09051933}, and \cref{12061911} are natural extensions of the corresponding results in \cite{stoev}.
\end{example}
It is well-known that stochastic integrals are often used for the representation of stochastic processes (or random fields, respectively), where the properties of the considered integral usually determine the spectrum of possible representations (see \cref{section1}). Hence, in view of \cref{09051933}, it should not surprise that we introduce the following notion which is due to Definition 3.0.1 in \cite{goldbach}. 
\begin{defi} \label{12071904}
Let $I \ne \emptyset$ be some index set. Then, an $\R^d$-valued stochastic process $\mathbb{X}=\{X(t):t \in I\}$ is called \textit{$f$-implicit max-stable} if, for all $k \in \N$ and $a_1,...,a_k \ge 0$ as well as $t_1,...,t_k \in I$, the random vector
\begin{equation*}
\xi:=\fmax{j=1}{k} a_j X(t_j)
\end{equation*}
is $f$-implicit max-stable (in the sense of \eqref{eq:08071902}).
\end{defi}
Finally, inspired by the example from \eqref{eq:12071933} and as some kind of an outlook, we state the following observation, which can be easily concluded from \cref{09051933} and the fact that $f$-implicit $\alpha$-Fr\'{e}chet distributions are always $f$-implicit max-stable ones (see \cref{section1} again). The details are left to the reader. 
\begin{prop}
Fix $I \ne \emptyset$ and let $M$ be an $f$-implicit $\alpha$-Fr\'{e}chet sup-measure as before. Moreover, consider a family $(g_t)_{t \in I} \subset L^{\alpha}_{+}(m)$ of functions and define $X(t):=I(g_t)$ for every $t \in I$. Then, the resulting process $\mathbb{X}=\{X(t):t \in I\}$ is $f$-implicit max-stable. More precisely, for $a_1,...,a_k \ge 0$ and $t_1,...,t_k \in I$ as above, we a.s. have that
\begin{equation*}
\fmax{j=1}{k} a_j X(t_j)= \int_E^{\vee_f} \left( \bigvee_{j=1}^k a_j g_{t_j}(s) \right) \, dM(s) \quad \sim \quad \Phi_{\alpha,\kappa}^f(\norm{\vee_{j=1}^k a_j g_{t_j}}_{\alpha}).
\end{equation*}
\end{prop}
\section*{Acknowledgement}
The author would like to emphasize that this paper is inspired by the fundamental results in \cite{goldbach}, which my former colleague Johannes Goldbach developed during his PhD time under the supervision of Hans-Peter Scheffler. Moreover, the author wants to thank Marco Oesting for many fruitful discussions that were particularly helpful in the context of \cref{03041905}.

\end{document}